\documentclass{amsart}
\usepackage{amsmath,amssymb, amscd, xy,amsthm, mathrsfs, nicefrac, tikz-cd}
\usepackage{amsrefs}
\usepackage[colorlinks=true]{hyperref}
\usepackage{cleveref}
\usepackage{amscd}
\usepackage{amsthm}
\usepackage{mathtools}
\usepackage{mathrsfs}
\usepackage{bm}
\usepackage[margin=2.93cm]{geometry}

\allowdisplaybreaks

\newtheorem{innercustomthm}{Theorem}
\newenvironment{customthm}[1]
  {\renewcommand\theinnercustomthm{#1}\innercustomthm}
  {\endinnercustomthm}

\newtheorem{theorem}{Theorem}[section]
\newtheorem{lemma}[theorem]{Lemma}
\newtheorem{proposition}[theorem]{Proposition}
\newtheorem{corollary}[theorem]{Corollary}

\newtheorem{conjecture}[theorem]{Conjecture}

\theoremstyle{definition}
\newtheorem{defn}[theorem]{Definition}

\theoremstyle{remark}
\newtheorem*{remark}{Remark}

\numberwithin{equation}{section}

\def\tr{\operatorname{tr}}

\newcommand{\HP}{\mathit{HP}}
\newcommand{\op}{\operatorname}
\newcommand{\SLC}{\operatorname{SL}(2, \mathbb{C})}

\title{$\SLC$ Floer cohomology for surgeries on some knots}

\begin{document}

\author[Ikshu Neithalath]{Ikshu Neithalath}
\thanks {The author was supported by NSF grants DMS-1708320 and DMS-2003488.}
\address {Centre for Quantum Mathematics, Syddansk Universitet, Campusvej 55, Odense M, Denmark 5230}
\email {ikshu@imada.sdu.dk}

\begin{abstract} We establish a relationship between $\HP(Y)$, the Abouzaid-Manolescu sheaf-theoretic $\SLC$ Floer cohomology, for $Y$ a surgery on a small knot in $S^3$ and Curtis' $\SLC$ Casson invariant. We use this to compute $\HP$ for most surgeries on two-bridge knots. We also compute $\HP$ for surgeries on two non-small knots, the granny and square knots. We provide a partial calculation of the framed sheaf-theoretic $\SLC$ Floer cohomology, $\HP_{\#}(Y)$, for surgeries on two-bridge knots and apply the data we obtain to show the non-existence of a surgery exact triangle for $\HP_{\#}$.
\end{abstract}

\maketitle

\section{Introduction}

In \cite{AM}, the authors defined a new invariant of closed, connected, orientable 3-manifolds $Y$ called sheaf-theoretic $\SLC$ Floer cohomology, denoted $\HP(Y)$. It is defined as the hypercohomology of a certain perverse sheaf on the character scheme $\mathscr{X}_{\text{irr}}(Y)$. The perverse sheaf comes from a description of this space as a complex Lagrangian intersection. In this paper, we compute this invariant for surgeries on two-bridge knots, the granny knot, and the square knot.

Let $K$ be a knot in $S^3$ and $S^3_{p/q}(K)$ its $p/q$ Dehn surgery. When $S^3\backslash K$ contains no closed, essential surfaces, we say that $K$ is a small knot. The calculation of $\HP(S^3_{p/q}(K))$ for $K$ a small knot and for most values of $p/q$ reduces to the $\operatorname{SL}(2,\mathbb{C})$ Casson invariant $\lambda_{SL(2,\mathbb{C})}$ as defined by Curtis \cite{curtis} and explored in her joint work with Boden \cite{BC}. Specifically, we have
\begin{customthm}{1}\label{smallknot}
Let $K\subset S^3$ be a small knot, and let $Y=S^3_{p/q}(K)$ denote $p/q$ surgery on $K$. If $p/q$ is not one of the finitely many boundary slopes of $K$, then $HP(Y)\cong \mathbb{Z}_{(0)}^{\lambda_{\SLC}(Y)}$.
\end{customthm}
\begin{remark}
We will often use the notation $A_{(k)}$ to denote a graded abelian group with $A$ in degree $k$. A more common notation for this is $A[-k]$.
\end{remark}
We combine this theorem with calculations of $\lambda_{\SLC}(Y)$ in the literature to produce explicit formulae for $\HP(Y)$ when $Y$ is a surgery on a two-bridge knot.

In \cite{AM}, the authors also define a framed version of sheaf-theoretic Floer cohomology denoted $\HP_{\#}(Y)$. It is defined as the hypercohomology of a certain perverse sheaf on the representation scheme of $Y$, $\operatorname{Hom}(\pi_1(Y),\SLC)$. We would like to compute the framed sheaf-theoretic Floer cohomology, $\HP_{\#}(S^3_{p/q}(K))$, for surgeries on knots. However, the representation schemes are usually not zero-dimensional and are often singular. So, we only give a formula for $\HP_{\#}$ of surgeries for which the character scheme is zero-dimensional, smooth, and does not contain non-abelian reducible representations.
\begin{customthm}{2}\label{repcalc}
Let $K$ be a knot and let $Y=S^3_{p/q}(K)$ denote the 3-manifold obtained from $p/q$ Dehn surgery on $K$. Let $p'=p$ for $p$ odd and $p'=\frac{p}{2}$ for $p$ even. Assume that the character scheme $\mathscr{X}_{\text{irr}}(Y)$ is zero-dimensional and smooth and no $p'$-th root of unity is a root of the Alexander polynomial of $K$. Then,
\begin{equation*}
\HP^*_{\#}(Y)= H^*(\text{pt})^{\oplus 2-\sigma(p)} \oplus H^{*+2}(\mathbb{CP}^1)^{\oplus \frac{1}{2}(|p|-2+\sigma(p))} \oplus H^{*+3} (\op{PSL}(2,\mathbb{C}))^{\oplus \lambda_{\SLC}(Y)},
\end{equation*}
where $\sigma(p)\in\{0,1\}$ is the parity of $p$.
\end{customthm}
We use this theorem in conjunction with the calculation of $\lambda_{\SLC}(Y)$ for surgeries on two-bridge knots in \cite{2bridge} to show that there do not exist exact triangles relating $\HP_{\#}$ for surgeries on two-bridge knots.

In light of Theorem \ref{smallknot}, we are interested in computing $\HP(S^3_{p/q}(K))$ when $K$ is not a small knot. The character schemes of such manifolds may have positive dimensional components, in which case the calculation of the $\SLC$ Casson invariant is insufficient to determine $\HP$. In fact, when $K=K_1\# K_2$ is a composite knot, we are guaranteed to have positive dimensional components. We provide a calculation of $\HP$ with $\mathbb{F}=\mathbb{Z}/2\mathbb{Z}$ coefficients for surgeries on the square and granny knots. Recall that the granny knot is the connected sum of two right-handed trefoils, whereas the square knot is a composite of a trefoil with its mirror.
\begin{customthm}{3}\label{HPgranny}
Let $S^3_{p/q}(3_1\# 3_1)$ denote the 3-manifold obtained from $p/q$ Dehn surgery on the granny knot, $3_1\# 3_1$. Then we have the following formula for the sheaf-theoretic Floer cohomology:
\begin{equation*}
    \HP(S^3_{p/q}(3_1\# 3_1);\mathbb{F})=
    \begin{cases}
	\mathbb{F}_{(0)}^{|6q-p|+\frac{1}{2}|12q-p|-\frac{3}{2}}\oplus \mathbb{F}_{(-1)}^{\frac{1}{2}|12q-p|-\frac{1}{2}} &\text{if $p$ is odd,}\ \\[10 pt]
	\mathbb{F}_{(0)}^{|6q-p|+\frac{1}{2}|12q-p|-1}\oplus \mathbb{F}_{(-1)}^{\frac{1}{2}|12q-p|-1}  &\text{if $p$ is even, $p\neq 12k$,}\ \\[10 pt]
	\mathbb{F}_{(0)}^{|6q-p|+\frac{1}{2}|12q-p|-5}\oplus \mathbb{F}_{(-1)}^{\frac{1}{2}|12q-p|+1}&\text{if $p=12k, p/q\neq 12$,}\ \\[10 pt]
	\mathbb{F}_{(1)}^4 \oplus\mathbb{F}_{(0)}^4 \oplus \mathbb{F}_{(-2)} &\text{if $p/q=12$.}\ \\
    \end{cases}
\end{equation*}
\end{customthm}
\begin{customthm}{4}\label{HPsquare}
Let $S^3_{p/q}(3_1\# 3_1^*)$ denote the 3-manifold obtained from $p/q$ Dehn surgery on the square knot, $3_1\# 3_1^*$ (where $3_1^*$ is the left-handed trefoil). Then we have the following formula for the sheaf-theoretic Floer cohomology:
\begin{equation*}
    \HP(S^3_{p/q}(3_1\# 3_1^*);\mathbb{F})=   
\begin{cases}
	\mathbb{F}_{(0)}^{\frac{1}{2}|6q-p|+\frac{1}{2}|6q+p|+\frac{1}2{}|p|-\frac{3}{2}}\oplus \mathbb{F}_{(-1)}^{\frac{1}{2}|p|-\frac{1}{2}} &\text{if $p$ is odd,}\ \\[10 pt]
	\mathbb{F}_{(0)}^{\frac{1}{2}|6q-p|+\frac{1}{2}|6q+p|+\frac{1}{2}|p|-1}\oplus \mathbb{F}_{(-1)}^{\frac{1}{2}|p|-1}  &\text{if $p$ is even, $p\neq 12k$,}\ \\[10 pt]
	\mathbb{F}_{(0)}^{\frac{1}{2}|6q-p|+\frac{1}{2}|6q+p|+\frac{1}{2}|p|-5}\oplus \mathbb{F}_{(-1)}^{\frac{1}{2}|p|+3} &\text{if $p=12k, p\neq 0$,}\ \\[10 pt]
	\mathbb{F}_{(1)}^4 \oplus\mathbb{F}_{(0)}^4 \oplus \mathbb{F}_{(-2)} &\text{if $p=0$.}\ \\
    \end{cases}
\end{equation*}
\end{customthm}

The organization of this paper is as follows. In Section 2 we provide some background on character varieties and the invariants $\HP$, $\HP_{\#}$, and $\lambda_{\SLC}$. In Section 3, we prove Theorem \ref{smallknot} and explain how it gives explicit formulae for the sheaf-theoretic $\SLC$ Floer cohomology of knot surgeries. In Section 4, we prove Theorem \ref{repcalc} and compute $\HP_{\#}$ for most surgeries on two-bridge knots. In Section 5, we determine the character variety of the composite knot $3_1\# 3_1$, allowing us to compute the A-polynomials of the square and granny knots in Section 6. In Section 7, we consider surgeries on composite knots and establish Theorems \ref{HPgranny} and \ref{HPsquare}. In Section 8, we apply our calculation of $\HP_{\#}$ of two-bridge knot surgeries to demonstrate the non-existence of a surgery exact triangle.

\textbf{Acknowledgements}. We have benefited from discussions with Laurent C\^{o}t\'{e}, Matt Kerr, Mohan Kumar, Jack Petok, Vivek Shende, and Burt Totaro. We also thank the two referees who provided detailed comments on the initial versions of this paper. We are particularly indebted to Ciprian Manolescu for his advice, support, and encouragement at all stages in the writing of this paper.

\section{Background}

For a topological space $X$, let $\mathscr{R}(X)$ denote the $\SLC$ representation scheme of $\pi_1(X)$, defined as
\begin{align*}
\mathscr{R}(X)=\operatorname{Hom}(\pi_1(X),\SLC).
\end{align*}
Assuming $\pi_1(X)$ is finitely generated, this set is naturally identified as the $\mathbb{C}$ points of an affine scheme. The character scheme $\mathscr{X}(X)$ is the GIT quotient of $\mathscr{R}(X)$ by the conjugation action of $\SLC$.

A representation $\rho\in\mathscr{R}(X)$ is irreducible if the image of $\rho$ is not contained in any proper Borel subgroup. The irreducible representations comprise the stable locus for the GIT action. Let $\mathscr{R}_{\text{irr}}(X)\subset\mathscr{R}(X)$ denote the open subscheme corresponding to irreducible representations, and similarly $\mathscr{X}_{\text{irr}}(X)\subset\mathscr{X}(X)$. When $X$ is a closed surface of genus $g>1$, $\mathscr{X}_{\text{irr}}(X)$ is a holomorphic symplectic manifold of dimension $6g-6$ \cite{goldman}.

To investigate character schemes of 3-manifolds, we take the perspective of \cite{AM} using Heegaard splittings. Let $Y=U_0\cup_{\Sigma}U_1$ be a Heegaard splitting of a closed, orientable, 3-manifold $Y$ into two handlebodies $U_0$ and $U_1$ with Heegaard surface $\Sigma$. Then $\mathscr{X}_{\text{irr}}(U_i)$ is a complex Lagrangian in $\mathscr{X}_{\text{irr}}(\Sigma)$ and $\mathscr{X}_{\text{irr}}(Y)=\mathscr{X}_{\text{irr}}(U_0)\cap\mathscr{X}_{\text{irr}}(U_1)$ is a Lagrangian intersection \cite{AM}.

In \cite{bussi}, the author applies the work of \cite{joyce} to define a perverse sheaf of vanishing cycles associated to any Lagrangian intersection in a holomorphic symplectic manifold. A perverse sheaf on a scheme $X$ is a certain type of object in $D_c^b(X)$, the bounded derived category of complexes of constructible sheaves on $X$. The category of perverse sheaves, $\operatorname{Perv}(X)$, is an abelian subcategory of $D_c^b(X)$. Perverse sheaves have wide application in algebraic geometry and are often used to study the topology of complex varieties. Given a function $f:U\to\mathbb{C}$ on a smooth scheme $U$, we can define a perverse sheaf of vanishing cycles, $\mathcal{PV}_f\in\operatorname{Perv}(U)$, with the property that the cohomology of the stalk of $\mathcal{PV}_f$ at a point $x$ is the cohomology of the Milnor fiber of $f$ at $x$ (up to a degree shift). The perverse sheaf associated to a Lagrangian intersection in \cite{bussi} is modeled on perverse sheaves of vanishing cycles.

In \cite{AM}, the authors use Bussi's construction to associate a perverse sheaf to a Heegaard splitting of a 3-manifold. Moreover, they show that the perverse sheaf is independent of the Heegaard splitting:

\begin{theorem}[\cite{AM}]
Let $Y$ be a closed, connected, oriented 3-manifold with a Heegaard splitting $Y=U_0\cup_{\Sigma} U_1$. Define the Lagrangians $L_i=\mathscr{X}_{\text{irr}}(U_i)\subset \mathscr{X}_{\text{irr}}(\Sigma)$. Apply the construction of \cite{bussi} to obtain a perverse sheaf $P_{L_0,L_1}\in\operatorname{Perv}(\mathscr{X}_{\text{irr}}(Y))$ associated to the Lagrangian intersection $\mathscr{X}_{\text{irr}}(Y)=L_0\cap L_1$. Then $P(Y):=P_{L_0,L_1}$ is an invariant of the 3-manifold $Y$ up to canonical isomorphism in $\operatorname{Perv}(\mathscr{X}_{\text{irr}}(Y))$.

\end{theorem}

We call its hypercohomology $\HP^*(Y)=\mathbb{H}^*(P(Y))$ the sheaf-theoretic $\SLC$ Floer cohomology of $Y$. They also define an invariant using the representation scheme that takes into account the reducibles, called the framed sheaf-theoretic $\SLC$ Floer cohomology of $Y$, $\HP_{\#}(Y)$. To define this invariant, we use the notion of the twisted character variety.

\begin{defn}
Let $\Sigma$ be a closed surface with a basepoint $w$. Let $D$ be a small disc neighborhood of $w$. We define the \emph{twisted character variety} as
\begin{align*}
\mathscr{X}_{\text{tw}}(\Sigma,w)=\{\rho\in\operatorname{Hom}(\pi_1(\Sigma-\{w\}),\SLC)\mid\rho(\partial D)=-I\}//\SLC.
\end{align*}
\end{defn}

The twisted character variety is a smooth, holomorphic symplectic manifold.

Given a Heegaard splitting $Y=U_0\cup_{\Sigma} U_1$ and a base point $z\in\Sigma$, we define $Y^{\#}=Y\#(T^2\times[0,1])$, where the connected sum is performed in a neighborhood of $z$, arranged so that $T^2\times [0,1/2]$ is attached to $U_0$ and $T^2\times [1/2,1]$ is attached to $U_1$. Let $\Sigma^{\#}=\Sigma\#(T^2\times[1/2])$ be the new splitting surface and $U_i^{\#}$ be the resulting compression bodies. Then, choose a basepoint $w\in T^2\times\{1/2\}$ away from the connected sum region. Let $\ell_0=w\times[0,1/2]$ and $\ell_1=w\times[1,1/2]$ be lines in each compression body.

In the holomorphic symplectic manifold $\mathscr{X}_{\text{tw}}(\Sigma^{\#},w)$, the subspaces $L_i^{\#}$ consisting of twisted representations that factor through $\pi_1(U_i^{\#}-\ell_i)$ are complex Lagrangian submanifolds. Furthermore, their intersection $L_0^{\#}\cap L_1^{\#}$ can be identified with the representation variety $\mathscr{R}(Y)$ \cite{AM}. Analogously to the previous situation, this leads to a perverse sheaf invariant of the 3-manifold, $\mathcal{P}_{\#}(Y)\in\operatorname{Perv}(\mathscr{R}(Y))$. We denote its hypercohomology $\HP_{\#}(Y)$, the framed sheaf-theoretic $\SLC$ Floer cohomology of $Y$.

To compute the invariants $\HP(Y)$ and $\HP_{\#}(Y)$, we can use the following proposition:
\begin{proposition}\label{smooth}
Let $X\subset \mathscr{X}_{\text{irr}}(Y)$ (resp. $X\subset \mathscr{R}_{\text{irr}}(Y)$) be a smooth topological component of the character scheme (resp. representation scheme) of complex dimension $d$. Then the restriction of the perverse sheaf $\mathcal{P}(Y)$ (resp. $\mathcal{P}_{\#}(Y)$) to $X$ is a local system with stalks isomorphic to $\mathbb{Z}[d]$. In particular, if $X$ is simply connected, then $\HP(Y)$ (resp. $\HP_{\#}(Y)$) contains $H^*(X)[d]$ as a direct summand.

Furthermore, if $[\rho]$ is an isolated irreducible character and $X\cong \op{PSL}(2,\mathbb{C})$ is the orbit of $[\rho]$ in the representation scheme, then the local system $P_{\#}(Y)\vert_X$ is trivial. 
\end{proposition}
\begin{proof}
The first part is Proposition 6.2 in \cite{AM}. The second part is Lemma 8.3 of \cite{AM}.
\end{proof}

When $X$ is smooth but not simply connected, then there is some ambiguity over the local system $P(Y)\vert_X$. This can be circumvented by using $\mathbb{Z}/2\mathbb{Z}$ coefficients.

\begin{corollary}\label{modtwo}
Assume $\mathscr{X}_{\text{irr}}(M)$ is smooth with topological components $X_i$ of complex dimensions $d_i$. Then $\HP(Y;\mathbb{Z}/2\mathbb{Z})=\bigoplus\limits_{i} H^*(X_i;\mathbb{Z}/2\mathbb{Z})[d_i]$.
\end{corollary}
\begin{proof}
This follows from the fact that all local systems with $\mathbb{Z}/2\mathbb{Z}$ coefficients are trivial, since $\operatorname{Aut}(\mathbb{Z}/2\mathbb{Z})$ is trivial.
\end{proof}

Morally, $\HP(Y)$ should be a version of (the dual of) instanton Floer homology using the gauge group $\SLC$ instead of $\operatorname{SU}(2)$. Pursuing this analogy, the Euler characteristic of $\HP(Y)$, denoted $\lambda^P(Y)$, should be a type of Casson invariant, just as the Euler characteristic of instanton Floer homology is related to the original Casson invariant, which is a count of irreducible $\operatorname{SU}(2)$ characters. There is another invariant called the $\SLC$ Casson invariant defined in \cite{curtis} that counts isolated, irreducible $\SLC$ characters. To distinguish it from this invariant, $\lambda^P(Y)$ is called the full Casson invariant since it takes into account the positive dimensional components of the character scheme. When $\mathscr{X}_{\text{irr}}(Y)$ is zero-dimensional, $\lambda^P$ and $\lambda_{\SLC}$ agree. In fact, we have

\begin{theorem}\label{thm:zerodim}
Let $Y$ be a closed, orientable 3-manifold such that $\mathscr{X}_{\text{irr}}(Y)$ is zero-dimensional. Then $\HP(Y)\cong \mathbb{Z}_{(0)}^{\lambda_{SL(2,\mathbb{C})}(Y)}$, where $\lambda_{SL(2,\mathbb{C})}(Y)$ is the $SL(2,\mathbb{C})$ Casson invariant as defined in \cite{curtis}.
\end{theorem}
\begin{proof}
The definition of $\HP(Y)$ uses the characterization of $\mathscr{X}_{\text{irr}}(Y)$ as a complex Lagrangian intersection $L_0\cap L_1$ in the character scheme of a Heegaard surface for $Y$. The stalk of the perverse sheaf $P^{\bullet}(Y)$ at a point $p\in \mathscr{X}_{\text{irr}}(Y)$ is the degree-shifted cohomology of the Milnor fiber of some function $f:U\to\mathbb{C}$, for $U$ an open neighborhood in one of the Lagrangians, such that the graph $\Gamma_{df}\subset T^*U$ is identified with $L_1$ in an appropriate polarization of the symplectic manifold near $p$. Since $\mathscr{X}_{\text{irr}}(Y)$ is zero-dimensional, we know that $f$ has an isolated singularity at $p$. Thus, the Milnor fiber has the homotopy type of a bouquet of spheres. The number of spheres in the bouquet is the Milnor number, denoted $\mu_p$. Then, the stalk is given by  $(P^{\bullet}(Y))_p\cong \mathbb{Z}_{(0)}^{\mu_p}$. The hypercohomology is $\HP(Y)\cong \mathbb{Z}_{(0)}^{\sum\mu_p}$, where the sum is over all components of $\mathscr{X}_{\text{irr}}(Y)$.
The definition of the Casson invariant in terms of intersection cycles given in \cite{curtis} is $\lambda_{SL(2,\mathbb{C})}(Y)=\sum_{p} n_p$, where the sum is over all zero-dimensional components of  $X_{\text{irr}}(Y)$, and $n_p$ is the intersection multiplicity of $L_0$ with $L_1$. But the Milnor number $\mu_p$ is equal to the intersection multiplicity of $\Gamma_{df}$ with $L_0$, hence the result follows.
\end{proof}

Theorem \ref{thm:zerodim} is useful when we can guarantee that $\mathscr{X}_{\text{irr}}(Y)$ is zero-dimensional. One case in which this holds is when $Y$ has no incompressible surfaces. 

\begin{corollary}\label{nsl}
Let $Y$ be a closed, orientable, not sufficiently large 3-manifold. Then $\HP(Y)\cong \mathbb{Z}_{(0)}^{\lambda_{SL(2,\mathbb{C})}(Y)}$.
\end{corollary}
\begin{proof}
By the main result of \cite{culler-shalen}, the character variety of a NSL 3-manifold is zero-dimensional. So, Theorem \ref{thm:zerodim} applies.
\end{proof}

\section{Surgeries on Small Knots and the $\lambda_{SL(2,\mathbb{C})}$ Casson Invariant}

\subsection{Surgeries on small knots}

By applying Theorem \ref{thm:zerodim}, we can establish the connection between $\HP(Y)$ for $Y$ a surgery on a small knot in $S^3$ and the $\SLC$ Casson invariant, $\lambda_{\SLC}(Y)$ as given in Theorem \ref{smallknot}.
\begin{proof}[Proof of Theorem \ref{smallknot}]
For any knot $K$, an incompressible surface in $Y=S^3_{p/q}(K)$ either comes from a closed essential surface in $S^3\backslash K$ or from an essential surface in $S^3\backslash K$ with boundary slope $p/q$ \cite{dehnknots}. When $K$ is small, the first case is ruled out. So, Corollary \ref{nsl} shows that $\HP(Y)\cong \mathbb{Z}_{(0)}^{\lambda_{SL(2,\mathbb{C})}(Y)}$ whenever $p/q$ is not a boundary slope. Note that there are only finitely many boundary slopes by \cite{hatcherboundary}. 
\end{proof}
The invariant $\lambda_{SL(2,\mathbb{C})}$ has been computed for a range of 3-manifolds, including surgeries on many families of knots \cite{curtis}\cite{2bridge}\cite{seifertfiber}. We provide a few examples of how those results yield formulae for the sheaf-theoretic Floer cohomology of surgeries on knots.

First we review the results of \cite{curtis} in order to consider the case of a general small knot. Let $M=S^3\backslash N(K)$ be a knot exterior.  Let $i:\partial M\to M$ denote the inclusion and $r:\mathscr{X}(M)\to \mathscr{X}(\partial M)$ denote the restriction map.
\begin{defn}
A slope $\gamma\in\partial M$ is \emph{irregular} if there exists an irreducible representation $\rho$ of $\pi_1(M)$ such that:

(i) the character $[\rho]$ is in a one-dimensional component $\mathcal{X}_i$ of $\mathscr{X}_{\text{irr}}(M)$ such that $r(\mathcal{X}_i)$ is also one-dimensional;

(ii) $\tr(\rho(\alpha))=\pm 2$ for all $\alpha\in \partial M$;

(iii) $\ker(\rho\circ i_*)$ is cyclic, generated by $[\gamma]$.
\end{defn}
\begin{defn}\label{admissible}
A slope $p/q$ is \emph{admissible} if:

(i) It is regular and not a strict boundary slope;

(ii) No $p'$-th root of unity is a root of the Alexander polynomial of $K$, where $p'=p$ for $p$ odd and $p'=\frac{p}{2}$ for $p$ even.
\end{defn}
With these definitions, we can state Theorem 4.8 of \cite{curtis}:
\begin{theorem}[\cite{curtis}]\label{thm:curtis}
Let $K$ be a small knot in $S^3$ with complement $M$. Let $\{\mathcal{X}_i\}$ be the collection of one-dimensional components of $\mathscr{X}(M)$ such that $r(\mathcal{X}_i)$ is one-dimensional and such that $\mathcal{X}_i$ contains an irreducible representation. Then there exist integral weights $m_i>0$ depending only on $\mathcal{X}_i$ and non-negative $E_0,E_1\in\frac{1}{2}\mathbb{Z}$ depending only on $K$ such that for every admissble $\frac{p}{q}$ we have
\begin{align*}
\lambda_{\operatorname{SL}(2,\mathbb{C})}(S^3_{p/q}(K))=\frac{1}{2}\sum\limits_{i} m_i ||p\mathscr{M}+q\mathscr{L}||_i-E_{\sigma(p)},
\end{align*}
where $\sigma(p)\in\{0,1\}$ is the parity, $||-||_i$ is the Culler-Shalen seminorm associated to $\mathcal{X}_i$ and $\sum\limits_{i} m_i ||p\mathscr{M}+q\mathscr{L}||_i:=||p/q||_T$ is the total Culler-Shalen seminorm.
\end{theorem}
There are only finitely many irregular slopes and only finitely many boundary slopes \cite{curtis}. Provided $p$ is chosen so that no $p'$-th root of unity is a root of the Alexander polynomial, where $p'$ is as in Definition \ref{admissible}(ii), the above theorem only excludes finitely many slopes $p/q$. Thus, by combining Theorem \ref{smallknot} with Theorem \ref{thm:curtis} we obtain a formula for the sheaf-theoretic Floer cohomology for most surgeries on small knots. 
\subsection{$\HP$ for surgeries on two-bridge knots}
Let $K(\alpha,\beta)$ be the two-bridge knot as defined by the notation in \cite{burdezieschang}. The $\SLC$ Casson invariants of surgeries on these small knots were computed in Theorem 2.5 of \cite{2bridge}. Applying their result and Theorem \ref{smallknot}, we obtain
\begin{proposition}\label{HP2bridge}
Let $S^3_{p/q}(K(\alpha,\beta))$ denote $p/q$ surgery on the two-bridge knot $K(\alpha,\beta)$. Assume $p/q$ is not a boundary slope and no $p'$-th root of unity is a root of the Alexander polynomial of $K(\alpha,\beta)$, where $p'=p$ for $p$ odd and $p'=p/2$ for $p$ even. Let $||p/q||_T$ denote the total Culler-Shalen seminorm of $p/q$. Then,
$$\HP(S^3_{p/q}(K(\alpha,\beta)))= \begin{cases} \mathbb{Z}_{(0)}^{\frac{1}{2}||p/q||_T} &\text{if $p$ is even,}\ \\[10 pt]
	 \mathbb{Z}_{(0)}^{\frac{1}{2}||p/q||_T-\frac{1}{4}(\alpha-1)} &\text{if $p$ is odd.}\ \\
    \end{cases}$$
\end{proposition}
\section{$\HP_{\#}$ for surgeries on two-bridge knots}
In this section, we prove Theorem \ref{repcalc} computing $\HP_{\#}(Y)$ when $Y=S^3_{p/q}(K)$ is a surgery on a knot $K$ (under some strong restrictions).

\begin{proof}[Proof of Theorem \ref{repcalc}]
We are assuming that no $p'$-th root of unity is a root of the Alexander polynomial of $K$, where $p'=p$ for $p$ odd and $p'=\frac{p}{2}$ for $p$ even. By Lemma \ref{nars}, this condition ensures that there are no non-abelian reducibles. We first consider the abelian representations. These representations are those which factor through $H_1(Y;\mathbb{Z})\cong\mathbb{Z}/p\mathbb{Z}$. So, we have 
\begin{align*}
\mathscr{R}_{\text{ab}}(Y)\cong \operatorname{Hom}(\mathbb{Z}/p\mathbb{Z},\SLC).
\end{align*}

Letting $a$ denote the generator of $\mathbb{Z}/p\mathbb{Z}$, we see that $\rho(a)$ can be any $\SLC$ matrix with eigenvalues $p$-th roots of unity. There are $|p|$ such roots. When $p$ is even, two of these roots are $\pm 1$, and $\pm I$ are the unique matrices with thoses eigenvalues. The other $|p|-2$ roots come in pairs $\zeta,\zeta^{-1}$ yielding conjugate $\SLC$ matrices. This gives $\frac{1}{2}(|p|-2)$ distinct conjugacy classes of matrices, each of which gives a conjugation orbit's worth of choices, which for an abelian, non-central representation is a copy of $T\mathbb{CP}^1$. Thus, we obtain 2 points and $\frac{1}{2}(|p|-2)$ copies of $T\mathbb{CP}^1$ in the representation variety. Similarly, for $p$ odd, there is only one central representation and $\frac{1}{2}(|p|-1)$ copies of $T\mathbb{CP}^1$.

For the irreducible representations, the Casson invariant $\lambda_{\SLC}(Y)$ gives the count of isolated points with multiplicity in the character variety. Since we are assuming that $\mathscr{X}_{\text{irr}}(Y)$ is zero-dimensional, the isolated points account for all irreducibles. Since we assume the scheme is smooth, the multiplicities of the points are all $1$ and the Casson invariant gives an honest count of points. The conjugation orbit of each isolated irreducible representation is a copy of $\operatorname{PSL}(2,\mathbb{C})$.

For an irreducible representation $\rho$, the character scheme is smooth at $[\rho]$ if and only if the representation scheme is smooth at $\rho$ by Lemma 2.4 in \cite{AM}. Since we are assuming that the character scheme is smooth, we conclude that the representation scheme is smooth. Hence, we can apply Proposition \ref{smooth} to compute $\HP_{\#}$.
\end{proof}
One situation in which we can apply Theorem \ref{repcalc} is when $K$ is a two-bridge knot.
\begin{proposition}\label{HPframed2bridge}
Let $K(\alpha,\beta)$ be a two-bridge knot and let $Y=S^3_{p/q}(K(\alpha,\beta))$ denote the 3-manifold obtained from $p/q$ Dehn surgery on $K$. Let $p'=p$ for $p$ odd and $p'=\frac{p}{2}$ for $p$ even. Assume that $p/q$ is not a boundary slope and that no $p'$-th root of unity is a root of the Alexander polynomial of $K$. Then,
\begin{equation*}
\HP^*_{\#}(Y)= H^*(\text{pt})^{\oplus 2-\sigma(p)} \oplus H^{*+2}(\mathbb{CP}^1)^{\oplus \frac{1}{2}(|p|-2+\sigma(p))} \oplus H^{*+3} (\op{PSL}(2,\mathbb{C}))^{\oplus (\frac{1}{2}||p/q||_T-\sigma(p)\frac{\alpha-1}{4})},
\end{equation*}
where $\sigma(p)\in\{0,1\}$ is the parity of $p$.
\end{proposition}
\begin{proof}
According to Proposition 2.2 of \cite{2bridge}, the zero-dimensional components of $\mathscr{X}_{\text{irr}}(S^3_{p/q}(K))$ are smooth under these hypotheses. Since $p/q$ is not a boundary slope, we know that $\mathscr{X}_{\text{irr}}(S^3_{p/q}(K))$ is entirely zero-dimensional by Culler-Shalen theory. Thus, Theorem \ref{repcalc} implies the result.
\end{proof}
In Section 8, we use this calculation to show that there does not exist an exact triangle relating $\HP_{\#}$ for surgeries on two-bridge knots.

\begin{remark}The hypothesis that the character scheme is zero-dimensional allows us to identify $\mathscr{R}_{\text{irr}}$ as some copies of $\operatorname{PSL}(2,\mathbb{C})$. In general, assuming $\mathscr{X}_{\text{irr}}$ is smooth, $\mathscr{R}_{\text{irr}}$ would be a fibration over $\mathscr{X}_{\text{irr}}$ with fibers diffeomorphic to $\operatorname{PSL}(2,\mathbb{C})$. Determining the precise structure of this fibration is not immediate. In Propositions \ref{grannyvar} and \ref{squarevar}, we compute the character schemes for the granny and square knots and show they are not a disjoint union of contractible components. So, one cannot immediately deduce the cohomology groups of $\mathscr{R}_{\text{irr}}$ (which give the interesting part of $\HP_{\#}$) from those of $\mathscr{X}_{\text{irr}}$.\end{remark}

\begin{section}{The character variety of $S^3\backslash (3_1\# 3_1)$}
The knot group of the trefoil has the presentations
\begin{align*}
\pi_1(S^3\backslash 3_1)=& \langle a,b\mid a^3=b^2\rangle, \\
\cong & \langle r,s \mid rsr=srs\rangle.
\end{align*}
In the first presentation, the meridian is given by $a^2b^{-1}$ and the longitude is given by $ba(b^{-1}a)^5$. In the second presentation, the meridian is $r$ and the longitude is $sr^2sr^{-4}$. The character scheme of $3_1$ is
\begin{align*}
\mathscr{X}(S^3\backslash 3_1)\cong \{(y-2)(x^2-y-1)=0\}\subset\mathbb{C}^2,
\end{align*}
where $x=\tr\rho(r)$ and $y=\tr(rs^{-1})$. The line $\{y=2\}$ is $\mathscr{X}_{\text{red}}$ and $\{x^2-y=1,y\neq 2\}$ is $\mathscr{X}_{\text{irr}}$. 

The fundamental group of the complement of the knot $3_1\#3_1$ (which is isomorphic to the knot group of $3_1\#3_1^*$) has the presentation
\begin{align*}
\Gamma=\langle a,b,c,d\mid a^3=b^2, c^3=d^2, d=ba^{-2}c^2\rangle,
\end{align*}
where the subgroup $\Gamma_0$ generated by $a$ and $b$ corresponds to a copy of $\pi_1(S^3\backslash 3_1)$ and similarly the subgroup $\Gamma_1$ generated by $c,d$ corresponds to the knot group of the other $3_1$ summand. The relation $a^2b^{-1}=c^2d^{-1}$ comes from setting the meridian in $\Gamma_0$ equal to the meridian in $\Gamma_1$. Consider the following closed subsets of $\mathscr{X}(\Gamma)$,
\begin{align*}
\mathscr{X}_{\operatorname{red}}= &\{[\rho]\mid \rho \text{ is abelian}\}\text{ and } \\
\mathscr{X}_{i}= &\{[\rho]\mid \rho\vert_{\Gamma_{|1-i|}} \text{ is abelian}\},
\end{align*}
where clearly $\mathscr{X}_{\operatorname{red}}\subset \mathscr{X}_{i}$. Since the abelianization of the knot group is generated by the meridian, we have that $\mathscr{X}_{\operatorname{red}}=\mathscr{X}(\mathbb{Z})\cong\mathbb{C}$, where the meridional trace is a coordinate for $\mathbb{C}$. 
\begin{lemma}
Let $\mathscr{X}(\Gamma)\overset{r}{\rightarrow} \mathscr{X}(\Gamma_{i})$ denote the natural restriction map. Then the composite $\mathscr{X}_i\hookrightarrow\mathscr{X}(\Gamma)\overset{r}{\rightarrow} \mathscr{X}(\Gamma_{i})$ is an isomorphism $\mathscr{X}_i\cong \mathscr{X}(\Gamma_{i})$.
\end{lemma}
\begin{proof}
If $\rho\vert_{\Gamma_{|1-i|}}$ is abelian, then it is determined by its value on the meridian. But the value of $\rho$ on the meridian is determined by its restriction to $\Gamma_{i}$, since the meridian lies in the intersection $\Gamma_0\cap\Gamma_1$, establishing injectivity.

For surjectivity, we observe that for any representation $\rho\in \mathscr{X}(\Gamma_i)$, there exists an extension of $\rho$ to a representation of $\Gamma$ given by setting $\rho\vert_{\Gamma_{|1-i|}}$ to be the abelian representation of $\Gamma_{|1-i|}$ with the required meridional value. This lies in $\mathscr{X}_i$ by construction.
\end{proof}
Recall the following general fact:
\begin{lemma}\cite{CCGSL}\label{nars}
Let $\rho$ be a representation of $\pi_1(S^3\backslash K)$ with $[\rho]\in \mathscr{X}_{\operatorname{red}}\cap\overline{\mathscr{X}_{\operatorname{irr}}}$. Then the following equivalent conditions hold:
\begin{itemize}
\item $\Delta(\mu^2)=0$, where $\Delta$ is the Alexander polynomial of $K$ and $\mu$ is an eigenvalue of $\rho(m)$, for $m$ the meridian of the knot.
\item There exists a non-abelian reducible representation $\rho'$ with the same character as $\rho$.
\end{itemize}
\end{lemma}
The Alexander polynomial of the trefoil is the sixth cyclotomic polynomial, $\Delta_{3_1}(t)=t^2-t+1$. Thus, the above lemma guarantees non-abelian reducibles at meridional trace $\pm\sqrt{3}$. The same holds for $3_1\# 3_1$ since $\Delta_{3_1\# 3_1}=(\Delta_{3_1})^2$. This allows us to establish the following proposition:
\begin{proposition}\label{components}
Let $\mathscr{X}_{i,\operatorname{irr}}=\mathscr{X}_{i}\cap \mathscr{X}_{\operatorname{irr}}$ and let $S=\mathscr{X}(\Gamma)\backslash (\mathscr{X}_0\cup\mathscr{X}_1)$. Then the four irreducible components of $\mathscr{X}(\Gamma)$ are $\mathscr{X}_{\operatorname{red}}, \overline{\mathscr{X}}_{0,\operatorname{irr}}, \overline{\mathscr{X}}_{1,\operatorname{irr}}$ and $\overline{S}$. Moreover, these four components pairwise intersect in the same two points, corresponding to characters of non-abelian reducibles. 
\end{proposition}
\begin{proof}
That none of the four closed sets share any irreducible components follows from the description of the intersections. If $[\rho]\in \overline{S}\cap \overline{\mathscr{X}}_{0,\operatorname{irr}}$, then by restricting to $\mathscr{X}(\Gamma_{1})$, we see that $[\rho\vert_{\Gamma_1}]\in \mathscr{X}_{\operatorname{red}}(\Gamma_1)\cap\overline{\mathscr{X}_{\operatorname{irr}}(\Gamma_1)}$. Thus, Lemma \ref{nars} implies that $[\rho]$ is one of two points in $\mathscr{X}_{\operatorname{red}}$ corresponding to non-abelian reducibles. The other intersections follow similarly.

It only remains to check that each of the four pieces is in fact irreducible. From the coordinate description of $\mathscr{X}(S^3\backslash 3_1)$, we see that $\mathscr{X}_{\operatorname{red}}=\{y=2\}$ and the $\overline{\mathscr{X}}_{i,\operatorname{irr}}$ are equal to $\{x^2-y=1\}$. In either case, they are isomorphic to $\mathbb{C}$. The irreducibility of $\overline{S}$ follows from Proposition \ref{cubic} below.
\end{proof}

\begin{proposition}\label{cubic}
$\overline{S}$ is an affine cubic surface with precisely two $A_1$ singularities at the points $S_{\operatorname{sing}}=\overline{S}\backslash S=\mathscr{X}_{\text{nar}}$, the two characters of non-abelian reducible representations.
\end{proposition}
\begin{proof}
Let $A=\rho(a)$, $B=\rho(b)$, etc. If $[\rho]\in S$, then $\rho\vert_{\Gamma_i}$ is non-abelian. But since $a^3=b^2$ is a central element of $\Gamma_0$, we must have $A^3=B^2=\pm I$. However, if $B^2=I$, then $B=\pm I$ and $\rho\vert_{\Gamma_0}$ would be abelian. Thus, we must have $A^3=B^2=-I$, and similarly $C^3=D^2=-I$ and $A,C\neq -I$. These equations are equivalent to $\operatorname{tr}(A)=\operatorname{tr}(C)=1$ and $\operatorname{tr}(B)=\operatorname{tr}(D)=0$. Now, since $d=ba^{-2}c^2$, we see that $D=BAC^{-1}$. Thus, we have the inclusion
\begin{align*}
S\subset\mathscr{S}=\{[\rho]\in\mathscr{X}(F_3)\mid \operatorname{tr}(A)=\operatorname{tr}(C)=1, \operatorname{tr}(B)=\operatorname{tr}(BAC^{-1})=0\},
\end{align*}
where $F_3$ is the free group generated by $a,b,c$. Also, any representation of $F_3$ that lies in $\mathscr{S}$ is a representation of $\Gamma$, so that $\mathscr{S}\subset\mathscr{X}(\Gamma)$. Since $S$ is open in $\mathscr{X}(\Gamma)$, $\overline{S}$ is a union of the irreducible components meeting $S$. Thus, $\overline{S}=\mathscr{S}$ provided $\mathscr{S}$ is irreducible.

So, we now turn to describing the algebraic set $\mathscr{S}$. Regarding $\mathscr{X}(F_3)$ as the character variety of the four-holed sphere, we see that $\mathscr{S}$ is a relative character variety; $\mathscr{S}$ is the locus of characters of $\pi_1(S^2-\{p_0,p_2,p_3,p_4\})$ with fixed traces along the four boundary circles. This relative character variety can be computed \cite{frickeklein} to be the affine cubic hypersurface in $\mathbb{C}^3$ given by the equation
\begin{align*}
f=x^2+y^2+z^2+xyz-z-2=0,
\end{align*}
where $x=\operatorname{tr}(AB), y=\operatorname{tr}(B^{-1}C)$ and $z=\operatorname{tr}(A^{-1}C)$. Furthermore, the reducible representations, which are the points in $\overline{S}\backslash S$, correspond to $(x,y,z)=(\pm \sqrt{3}, \mp\sqrt{3}, 2)$. These are precisely the singular points of the affine cubic surface $\overline{S}$. Since the Tjurina number, $\dim\widehat{\mathcal{O}}_{(x,y,z)}/(f,\partial_x f,\partial_y f,\partial_z f)$, is equal to 1 at the singularities, they are $A_1$ singularities.
\end{proof}
We record a calculation of the singular cohomology groups of $S$ for use in Section 7.
\begin{proposition}\label{cohomology}
The singular cohomology groups of $S$ are
\begin{equation*}
    H^*(S;\mathbb{Z})=
    \begin{cases}
	\mathbb{Z} & i=0,\ \\
	0 & i=1,\ \\
	\mathbb{Z}^2 & i=2,\ \\
	\mathbb{Z}^4 & i=3,\ \\
	0 & i\geq 4.\ 
    \end{cases}
\end{equation*}
\end{proposition}
\begin{proof}
Let $Q$ denote the projective closure of $\overline{S}$ inside of $\mathbb{P}^3$. One can check that $Q$ is smooth at infinity, meaning that $Q_{\operatorname{sm}}$, the smooth locus of $Q$, is the complement of the two singularities at $\overline{S}\backslash S$. By Theorem 4.3 in \cite{dimcasingularities}, the homology groups of $Q$ are
\begin{equation*}
    H_*(Q;\mathbb{Z})=
    \begin{cases}
	\mathbb{Z} & i=0,\ \\
	0 & i=1,\ \\
	\mathbb{Z}^5 & i=2,\ \\
	0 & i=3,\ \\
	\mathbb{Z} & i=4.\ 
    \end{cases}
\end{equation*}
By Poincar\' {e} duality, 
\begin{align*}
H_n(Q_{\operatorname{sm}};\mathbb{Z})\cong H_c^{4-n}(Q_{\operatorname{sm}};\mathbb{Z}).
\end{align*}
We can equate the compactly supported cohomology with a relative cohomology group, 
\begin{align*}
H_c^{n}(Q_{\operatorname{sm}};\mathbb{Z})\cong H^{n}(Q,Q_{\operatorname{sing}};\mathbb{Z}),
\end{align*}
which can be determined from the long exact sequence
\begin{align*}
\dots \to H^n(Q,Q_{\operatorname{sing}};\mathbb{Z})\to  H^n(Q;\mathbb{Z})\to  H^n(Q_{\operatorname{sing}};\mathbb{Z})\to\dots
\end{align*}
In particular, since $Q_{\operatorname{sing}}$ is zero-dimensional, we see that $H_n(Q_{\operatorname{sm}};\mathbb{Z})\cong H^{4-n}(Q;\mathbb{Z})$ for $n\leq 2$. And
\begin{align*}
\operatorname{rk} H_3(Q_{\operatorname{sm}};\mathbb{Z})= & \operatorname{rk} H^1(Q,Q_{\operatorname{sing}};\mathbb{Z}),\\
=& \operatorname{rk} H^1(Q;\mathbb{Z})+|Q_{\operatorname{sing}}|-1.
\end{align*}
So, the homology groups of $Q_{\operatorname{sm}}$ are
\begin{equation*}
    H_*(Q_{\operatorname{sm}}; \mathbb{Z})=
    \begin{cases}
	\mathbb{Z} & i=0,\ \\
	0 & i=1,\ \\
	\mathbb{Z}^5 & i=2,\ \\
	\mathbb{Z} & i=3,\ \\
	0 & i\geq 4.\ 
    \end{cases}
\end{equation*}
Let $Q_{\infty}=Q\backslash \overline{S}$. Then $S=Q_{\operatorname{sm}}\backslash Q_{\infty}$. We have $Q_{\infty}=\{xyz=0\}\subset\mathbb{P}^2$, which is a triangular arrangement of three lines. The normal bundle of each of these three copies of $\mathbb{P}^1$ has degree $-1$. So, a neighborhood of each sphere inside of $S$ is diffeomorphic to the $D^2$ bundle over $S^2$ with Euler number $-1$. The boundary of this neighborhood is diffeomorphic to $S^3$. Hence, the boundary of a neighborhood of $Q_{\infty}$, $\partial N(Q_{\infty})$, is a necklace of three copies of $S^3$. We can then apply the Mayer-Vietoris sequence
\begin{align*}
\dots \longrightarrow H_*(\partial N(Q_{\infty})) \longrightarrow H_*(Q_{\infty})\oplus H_*(S) \longrightarrow H_*(Q_{\operatorname{sm}}) \longrightarrow \dots 
\end{align*}
to compute the stated cohomology groups.
\end{proof}
\end{section}

\begin{section}{The A-polynomials of the square and granny knots}
We wish to describe the image of the natural map $r: \mathscr{X}(\Gamma)\to \mathscr{X}(\partial(S^3\backslash (3_1 \# 3_1^\circ)))$ given by restriction to the boundary torus. We use the notation $3_1 \# 3_1^\circ$ to collectively refer to either $3_1 \# 3_1$ or $3_1 \# 3_1^*$. Coordinates on $\mathscr{X}(\partial S^3\backslash N(3_1 \# 3_1^\circ))=\mathscr{X}(T^2)$ are given by the traces of the meridian and longitude. One may consider the double branched cover $d:\mathbb{C}^*\times\mathbb{C}^*\to \mathscr{X}(T^2)$ where the coordinates on the cover are given by the eigenvalues of the meridian and longitude, $M$ and $L$. The definining polynomial for the closure of the pull-back of the image of $r$ to $\mathbb{C}^*\times\mathbb{C}^*$ is called the A-polynomial \cite{CCGSL}.

For the right-handed trefoil, the A-polynomial is $M^{-6}+L=0$, whereas for the left-handed trefoil it is $M^6+L=0$ \cite{CCGSL}. These equations define the image under $r$ of the components $\overline{\mathscr{X}_{i,irr}}$. $\mathscr{X}_{\text{red}}$ is mapped to the line $L=1$. 

\begin{lemma}\label{compapoly}
Let $S$ be as in Proposition \ref{components}. Then the defining equation of the algebraic set $d^{-1}(\overline{r(S)})$ in eigenvalue coordinates  is $L-M^{-12}=0$ for the granny knot (the composite of two right-handed trefoils) and $L=1$ for the square knot (the composite of oppositely oriented trefoils).
\end{lemma}
\begin{proof}
Let $\ell_i$ denote the longitude of the $i$-th summand of $3_1\# 3_1^\circ$. Then, the longitude of $3_1\# 3_1^\circ$ is $\ell=\ell_0\ell_1$. Also, each of the $\ell_i$ commutes with the meridian $\mu$ in $\Gamma$. Since $\rho(\mu)$ is non-central, this means that $\rho(\ell_0)$ and $\rho(\ell_1)$ must commute with each other. In fact, for the irreducible representations of the right-handed trefoil, we have $\rho(\ell_i)=-\rho(m)^{-6}$ and similarly $\rho(\ell_i)=-\rho(m)^{6}$ for the left-handed trefoil.

For $\rho\in S$, we have that $\rho$ restricted to either summand is irreducible. So for the granny knot, we then have $\rho(\ell)=\left(-\rho(m)^{-6}\right)^2=\rho(m)^{-12}$ and for the square knot we obtain $\rho(\ell)=1$. These matrix equations give the desired eigenvalue equations.
\end{proof}
\begin{proposition}
The A-polynomial of the granny knot, $3_1\# 3_1$, is
\begin{align*}
A_{3_1\# 3_1}=(L-1)(L+M^{-6})(L-M^{-12}).
\end{align*}
The A-polynomial of the square knot, $3_1\# 3_1^*$, is
\begin{align*}
A_{3_1\# 3_1^*}=(L-1)(L+M^{-6})(L+M^{6}).
\end{align*}
\end{proposition}
\begin{proof}
The A-polynomial is a product (omitting repeated factors) of the defining polynomials for the images of the four components of $\mathscr{X}(\Gamma)$. Two of the components are copies of $\mathscr{X}(3_1)$, and therefore contribute factors corresponding to the A-polynomial of right or left-handed trefoil. The reducibles give the factor of $L-1$. The factor coming from the two-dimensional component $\overline{S}$ was determined in Lemma \ref{compapoly}.
\end{proof}
We will also be interested in the defining equation for the image of the map $r: \mathscr{X}_{\text{irr}}(\Gamma)\to \mathscr{X}(T^2)$, where we only consider the irreducibles. Let us call the defining polynomial for this curve $A_{K}^{\text{irr}}(M,L)$. Then, by the above discussion, we find:
\begin{align*}
A_{3_1\# 3_1}^{\text{irr}}= &(L+M^{-6})(L-M^{-12}),\\
A_{3_1\# 3_1^*}^{\text{irr}}= &(L+M^{-6})(L+M^{6})(L-1).
\end{align*}
\end{section}

\section{Surgeries on the Granny and Square knots}
In this section, we prove Theorems \ref{HPgranny} and \ref{HPsquare}. We proceed by calculating the relevant character schemes, showing they are smooth, and then computing their singular cohomology groups so that we can apply Corrolary 2.2 to write $\HP$ as the (degree-shifted) singular cohomology of the character scheme.
\subsection{Character scheme of a composite knot}
First, we establish a general procedure for computing the (set-theoretic) characters of the exterior of a composite knot. Although the character variety of $3_1\# 3_1$ was computed in Section 5, the description given here will be particularly amenable for computing the character varieties of the surgeries. The description from Section 5 will also be useful.

Let $K_1$ and $K_2$ be two oriented knots in $S^3$ and set $K=K_1\# K_2$, $M_i=S^3\backslash K_i$, $M=S^3\backslash K$. We have the following pushout diagram of spaces:
\[
\begin{tikzcd}
M \arrow[hookleftarrow]{r}\arrow[hookleftarrow]{d}
  & M_1 \arrow[hookleftarrow]{d}{i_1} \\
M_2 \arrow[hookleftarrow]{r}{i_2}
  & S^1
\end{tikzcd}
\]
where $i_j(S^1)= m_j$, a meridian for $K_j$, $j=1,2$. By the Van Kampen theorem, we have the pushout diagram of groups:
\[
\begin{tikzcd}
\pi_1(M) \arrow[hookleftarrow]{r}\arrow[hookleftarrow]{d}
  & \pi_1(M_1) \arrow[hookleftarrow]{d}\\
\pi_1(M_2) \arrow[hookleftarrow]{r}
  & \pi_1(S^1)
\end{tikzcd}
\]
That is, $\pi_1(M)\cong \pi_1(M_1)*\pi_1(M_2)/\langle m_1= m_2 \rangle$. We have a pullback diagram of representation spaces:

\[
\begin{tikzcd}
\mathscr{R}(M) \arrow[twoheadrightarrow ]{r}\arrow[twoheadrightarrow]{d}
  & \mathscr{R}(M_1) \arrow[twoheadrightarrow]{d}{r_1}\\
\mathscr{R}(M_2) \arrow[twoheadrightarrow]{r}{r_2}
  & \mathscr{R}(S^1)
\end{tikzcd}
\]
To analyze $\mathscr{X}(M)=\mathscr{R}(M)//\SLC$, we can compare it to a simpler object: the fiber product of the character schemes $\mathscr{X}(M_1)\times_{\mathscr{X}(S^1)} \mathscr{X}(M_2)$. We have the diagram
\[
\begin{tikzcd}
\mathscr{X}(M)
\arrow[bend left]{drr}
\arrow[bend right,swap]{ddr}
\arrow{dr}[description]{\varphi} & & \\
& \mathscr{X}(M_1)\times_{\mathscr{X}(S^1)} \mathscr{X}(M_2) \arrow[twoheadrightarrow ]{r}\arrow[twoheadrightarrow]{d}
  & \mathscr{X}(M_1) \arrow[twoheadrightarrow]{d}{\overline{r}_1}\\
& \mathscr{X}(M_2) \arrow[twoheadrightarrow]{r}{\overline{r}_2}
  & \mathscr{X}(S^1)
\end{tikzcd}
\]
where $\overline{r}_1([\rho_1])=\tr(\rho_1( m_1))$.
\subsubsection*{Pullbacks and quotients}
In order to understand the character scheme of $M$ from the fiber product of the character schemes of $M_1$ and $M_2$, we must determine the pre-images of points under $\varphi$. We establish the following lemma:
\begin{lemma}\label{pullbackquot}
Let $\varphi: \mathscr{X}(M)\to \mathscr{X}(M_1)\times_{\mathscr{X}(S^1)} \mathscr{X}(M_2)$ denote the natural map as above. Then for any $p=([\rho_1],[\rho_2])\in \mathscr{X}(M_1)\times_{\mathscr{X}(S^1)} \mathscr{X}(M_2)$, we have
\begin{align*}
\varphi^{-1}(p)\cong \operatorname{Stab}(m)/\langle \operatorname{Stab}(\rho_1), \operatorname{Stab}(\rho_2)\rangle,
\end{align*}
where $m=r_1(\rho_1)=r_2(\rho_2)$
\end{lemma}
\begin{proof}
The pre-image of $p$ in $\mathscr{R}(M_1)\times \mathscr{R}(M_2)$ is $\operatorname{Orb}(\rho_1)\times \operatorname{Orb}(\rho_2)$. The pair $(\rho_1,\rho_2)$ is a point here that is also in $\mathscr{R}(M_1)\times_{\mathscr{R}(S^1)} \mathscr{R}(M_2)$. All other such points can be obtained by using the action of $\operatorname{Stab}(m)$ on each factor, or else using the diagonal action of $\SLC$. This gives the set 
\begin{align*}
 (\mathscr{R}(M_1)\times_{\mathscr{R}(S^1)} \mathscr{R}(M_2))\cap(\operatorname{Orb}(\rho_1)\times \operatorname{Orb}(\rho_2)) = &  \SLC\cdot (\operatorname{Stab}(m)\cdot \rho_1\times \operatorname{Stab}(m)\cdot\rho_2),\\
= & \SLC\cdot (\operatorname{Stab}(m)\cdot \rho_1\times \rho_2).
\end{align*}
Reducing modulo the diagonal action of $\SLC$,
\begin{align*}
& \SLC\cdot (\operatorname{Stab}(m)\cdot \rho_1\times \rho_2)/G,\\
= & \operatorname{Stab}(m)/\langle\operatorname{Stab}(\rho_1), \operatorname{Stab}(\rho_2)\rangle.
\end{align*}
Thus, $\varphi^{-1}(p)\cong \operatorname{Stab}(m)/\langle \operatorname{Stab}(\rho_1), \operatorname{Stab}(\rho_2)\rangle$.
\end{proof}
\subsection{Irreducible representations in the character scheme of a composite knot}
To determine the locus of irreducible representations $\mathscr{X}_{\text{irr}}(M)$, we first describe $\mathscr{X}(M_1)\times_{\mathscr{X}(S^1)}\mathscr{X}(M_2)$ and then use Lemma \ref{pullbackquot} to understand the fibers of $\varphi$ over the various components.

Recall that $\mathscr{X}(M)$ has a stratification $\mathscr{X}_{\text{nar}}\subset \mathscr{X}_{\text{red}}\subset \mathscr{X}$, where $\mathscr{X}_{\text{nar}}$ is the locus of characters of non-abelian reducible representations. The complement $\mathscr{X}_{\text{irr}}=\mathscr{X}\backslash \mathscr{X}_{\text{red}}$ is the locus of irreducibles. The scheme $\mathscr{X}_{\text{nar}}$ can be identified from Lemma 4.2. The characters of non-abelian reducibles are also the characters of abelian reducibles. That is, every reducible character has an associated orbit of abelian representations, but for those characters in  $\mathscr{X}_{\text{nar}}$, there is an additional orbit corresponding to non-abelian reducible representations.

Taking the product stratification on $\mathscr{X}(M_1)\times_{\mathscr{X}(S^1)}\mathscr{X}(M_2)$ gives nine different strata of six essentially different types. The following proposition states which strata intersect the image $\varphi(\mathscr{X}_{\text{irr}}(M))$ and also identifies the set of irreducible representations in the fiber of $\varphi$ over a point in a given stratum.

\begin{proposition}\label{compositechar}
Using the previously established notation, $\varphi(\mathscr{X}_{\text{irr}}(M))$ consists of the following pieces
\begin{itemize}
\item $\mathscr{X}_{\text{irr}}(M_1)\times_{\mathscr{X}(S^1)}\mathscr{X}_{\text{irr}}(M_2),$
\item $\mathscr{X}_{\text{irr}}(M_i),$
\item $\mathscr{X}_{\text{nar}}(M_1)\times_{\mathscr{X}(S^1)}\mathscr{X}_{\text{nar}}(M_2).$
\end{itemize}
The fibers of $\varphi$ are copies of: 
\begin{itemize}
\item $\mathbb{C}^*$ over points in $\mathscr{X}_{\text{irr}}(M_1)\times_{\mathscr{X}(S^1)}\mathscr{X}_{\text{irr}}(M_2)$ with meridional eigenvalue $\mu\neq \pm 1$.
\item $\mathbb{C}$ over points in $\mathscr{X}_{\text{irr}}(M_1)\times_{\mathscr{X}(S^1)}\mathscr{X}_{\text{irr}}(M_2)$ with meridional eigenvalue $\mu=\pm 1$.
\item  A single point over points in $\mathscr{X}_{\text{irr}}(M_i)$ with $\Delta(\mu^2)\neq 0$.
\item  $\mathbb{C}$ over points in $\mathscr{X}_{\text{irr}}(M_i)$ with $\Delta(\mu^2)= 0$.
\item  $\mathbb{C}^*\backslash\{1\}$ over points in $\mathscr{X}_{\text{nar}}(M_1)\times_{\mathscr{X}(S^1)}\mathscr{X}_{\text{nar}}(M_2)$.
\end{itemize}
\end{proposition}
\begin{proof}
First, we identify the copy of $\mathscr{X}_{\text{irr}}(M_1)$ that appears in $\mathscr{X}(M_1)\times_{\mathscr{X}(S^1)}\mathscr{X}(M_2)$. A reducible character is the character of an abelian representation, and the meridian generates the abelianization of the knot group. Thus, the isomorphism $H_1(M_1)\cong \pi_1(S^1)$, where $S^1$ a meridional circle, yields an isomorphism $\mathscr{X}_{\text{red}}(M_1)\cong \mathscr{X}(S^1)$. And taking fiber products, $\mathscr{X}_{\text{irr}}(M_1)\times_{\mathscr{X}(S^1)}\mathscr{X}_{\text{red}}(M_2)\cong \mathscr{X}_{\text{irr}}(M_1)$. 

Now, we show that image of $\varphi$ consists of the stated pieces. Indeed, the only strata not included in the list are contained in $(\mathscr{X}_{\text{red}}(M_1)\times_{\mathscr{X}(S^1)}\mathscr{X}_{\text{red}}(M_2))\backslash (\mathscr{X}_{\text{nar}}(M_1)\times_{\mathscr{X}(S^1)}\mathscr{X}_{\text{nar}}(M_2))$. These correspond to representations of the form $\rho_1*\rho_2$ where (e.g.) $\rho_1$ is abelian and $\rho_2$ is reducible. However, for an abelian representation, $\operatorname{im}(\rho_1)=\operatorname{im}(\rho_1\vert_{m_1})$ since the meridian $m_1$ generates the abelianization of $\pi_1(M_1)$. Thus, since the $\rho_i$ agree on $m_i$, we see that $\operatorname{im}(\rho_1*\rho_2)=\operatorname{im}(\rho_2)$, so that the composite representation is also reducible. Thus, none of these pairings provide irreducible representations. 

For $p=([\rho_1],[\rho_2])\in \mathscr{X}(M_1)\times_{\mathscr{X}(S^1)} \mathscr{X}(M_2)$, if both $[\rho_1],[\rho_2] \in \mathscr{X}_{\text{irr}}$, then $\operatorname{Stab}(\rho_i)=\{\pm 1\}$. Furthermore, $r_1(\rho_1)$ is an abelian, non-central representation (if $\rho_1( m)=\pm I$, then the entire representation is central because $ m_1$ normally generates $\pi_1(M_1)$). Thus, $\operatorname{Stab}(r_1(\rho_1))\cong\mathbb{C}^*$ for meridional trace not $\pm 2$, and $\operatorname{Stab}(r_1(\rho_1))\cong\mathbb{C}\times\mathbb{Z}/2$ otherwise. So, $\varphi^{-1}(p)\cong\mathbb{C}^*/\{\pm 1\}\cong\mathbb{C}^*$ or $\varphi^{-1}(p)\cong\mathbb{C}$ by Lemma \ref{pullbackquot}.

If $[\rho_1]$ is irreducible but $[\rho_2]$ is reducible, then we can find an abelian lift $\rho_2$, so that $\operatorname{Stab}(\rho_2)=\operatorname{Stab}(r_2(\rho))$, and the fiber $\varphi^{-1}(p)$ is a point. For a non-abelian lift of $\rho_2$, $\operatorname{Stab}(\rho_2)$ is trivial. Moreover, the trace of the meridian cannot be $\pm 1$ for a non-abelian reducible because $\Delta(\pm 1)\neq 0$. Therefore, the stabilizer of the meridian must be $\mathbb{C}^*$. The abelian lies in the closure of the orbit of non-abelian reducibles, so that $\varphi^{-1}(p)=\mathbb{C}$ for such a point.

If both are reducible and at least one is abelian, then the overall representation is reducible. If both are non-abelian reducibles, then the stabilizers of each representation are trivial and the stabilizer of the meridian is $\mathbb{C}^*$, giving that the fiber of $\varphi$ is $\mathbb{C}^*$. However, not all of these representations are irreducible. We have that $\operatorname{im}(\rho_i)\subset B_i$, for $B_1,B_2$ Borel subgroups. For some $d\in \operatorname{Stab}(r_1(\rho_1))$, the composite representation corresponding to $d$ has image generated by $\langle \operatorname{im}(\rho_1), d^{-1}\operatorname{im}(\rho_2)d\rangle$. If this image were contained in some Borel subgroup $B$, then $\operatorname{im}(\rho_1)$ would be contained in two Borel subgroups, so either it is contained in a diagonal subgroup (but then $\rho_1$ is abelian), or else $B=B_1$. Then, we have $d^{-1}\operatorname{im}(\rho_2)d\subset B_1$, and so by the same argument we conclude $B_1=d^{-1}B_2d$. Thus, $d\in\operatorname{Stab}(B_2)$, which is trivial in $G^{\operatorname{ad}}$. Hence, precisely one point in $\operatorname{Stab}(r_1(\rho_1))$ corresponds to a reducible --- the rest are irreducible. So, the irreducibles in $\varphi^{-1}(p)$ form a copy of $\mathbb{C}^* - \{1\}$.
\end{proof}
\subsection{Character scheme of a connected sum of two trefoils}
We now focus on the case when $K_1=K_2=3_1$. The character scheme of the trefoil can be described as a plane curve:
\begin{align*}
\mathscr{X}(3_1)\cong \{(y-2)(x^2-y-1)=0\}\subset\mathbb{C}^2,
\end{align*}
where $x$ is the trace of the meridian. In terms of the Wirtinger presentation, we have $x=\tr(\rho(r))=\tr(\rho(s))$ and $y=\operatorname{tr}(rs^{-1})$. The line $\{y=2\}$ is $\mathscr{X}_{\text{red}}$ and $\{x^2-y=1, y\neq 2\}$ is $\mathscr{X}_{\text{irr}}$. The map $\overline{r}_1$ is projection onto the $x$ coordinate. The longitude for $3_1$ is $\ell=sr^2sr^{-4}$, and its trace in the $x,y$ coordinates is given by the polynomial
\begin{align*}
L(x,y) =  x^6y-2x^6-x^4y^2-2x^4y+8x^4+2x^2y^2+x^2y-10x^2+2.
\end{align*}
The restriction of $L(x,y)$ to $y=2$ is the constant function $2$, as expected. On this component, $\rho(\ell)=I$. The restriction of $L(x,y)$ to $y=x^2-1$ is $L=-x^6+6x^4-9x^2+2$, which can be deduced from the fact that for the irreducible representations, we have $\rho(\ell)=-\rho( m)^{-6}$.

The Alexander polynomial has roots that are primitive 6th roots of unity. So, non-abelian reducibles occur at the points $(\pm \sqrt{3}, 2)\in \mathscr{X}_{\text{red}}$. Observe that this is precisely $\overline{\mathscr{X}_{\text{irr}}}-\mathscr{X}_{\text{irr}}$.

The fiber product of the character varieties over the meridional trace map is 
\begin{align*}
\mathscr{X}(3_1)\times_{\mathbb{C}} \mathscr{X}(3_1)\cong \{(y-2)(x^2-y-1)=0,(z-2)(x^2-z-1)=0\}\subset\mathbb{C}^3.
\end{align*} 
Applying Proposition \ref{compositechar}, we have the following explicit descriptions of the fibers of $\varphi$ over points in the various strata of $\varphi(\mathscr{X}_{\text{irr}}(K_1\# K_2))$:
\begin{itemize}
\item $\mathscr{X}_{\text{irr}}\times_{\mathbb{C}} \mathscr{X}_{\text{irr}}= \{x^2-y-1=0,x^2-z-1=0,y\neq 2,z\neq 2\}$. The fibers of $\varphi$ are $\mathbb{C}^*$ unless $x=\pm 2$, in which case they are $\mathbb{C}$. 
\item $\mathscr{X}_{\text{irr}}(M_1)=\{z=2,x^2-y-1=0,y\neq 2\}$. Note that since $y\neq 2$, we have $x\neq \pm\sqrt{3}$ and $\Delta(m^2)\neq 0$. So, the fibers of $\varphi$ are just points. The same holds for $\mathscr{X}_{\text{irr}}(M_2)=\{y=2,x^2-z-1=0,z\neq 2\}$.
\item $\mathscr{X}_{\text{nar}}\times_{\mathbb{C}}\mathscr{X}_{\text{nar}}=\{(\pm \sqrt{3}, 2,2)\}$. The fibers of $\varphi$ are $\mathbb{C}^*\backslash\{1\}$.
\end{itemize}
\begin{remark}
To compare this description with that of Proposition \ref{components}, we see that 
\begin{itemize}
\item $\varphi^{-1}(\mathscr{X}_{\text{irr}}\times_{\mathbb{C}} \mathscr{X}_{\text{irr}}\cup \mathscr{X}_{\text{nar}}\times_{\mathbb{C}}\mathscr{X}_{\text{nar}})=S$,
\item $\mathscr{X}_{\text{irr}}(M_i)=\mathscr{X}_{i,irr}$.
\end{itemize}
\end{remark}
Since $\pi_1(3_1\# 3_1)$ and $\pi_1(3_1\# 3_1^*)$ are isomorphic, the same description applies to $\mathscr{X}_{\text{irr}}(3_1\# 3_1^*)$.
\subsection{Character scheme for granny knot surgeries}
Let $S^3_{p/q}(3_1\# 3_1)$ denote the $p/q$ surgery on the granny knot. We have the following description of $\mathscr{X}_{\text{irr}}(S^3_{p/q}(3_1\# 3_1))$.
\begin{proposition}\label{grannyvar}
$\mathscr{X}_{\text{irr}}(S^3_{p/q}(3_1\# 3_1))$ consists of $2\lambda_{\SLC}(S^3_{p/q}(3_1))$ points and
\begin{itemize}
\item $\lambda_{\SLC}(S^3_{p/2q}(3_1))$ copies of $\mathbb{C}^*$ when $p$ is odd.
\item $\lambda_{\SLC}(S^3_{p/2q}(3_1))-1$ copies of $\mathbb{C}^*$ when $p$ is even, $p\neq 12k$.
\item $\lambda_{\SLC}(S^3_{p/2q}(3_1))-3$ copies of $\mathbb{C}^*$ and 2 copies of $\mathbb{C}^*\backslash\{1\}$ when $p=12k, p/q\neq 12$.
\item $S=\varphi^{-1}(\mathscr{X}_{\text{nar}}\times_{\mathbb{C}}\mathscr{X}_{\text{nar}}\cup\mathscr{X}_{\text{irr}}\times_{\mathbb{C}}\mathscr{X}_{\text{irr}})$ when $p/q=12$.
\end{itemize}
\end{proposition}
We will describe $\mathscr{X}_{\text{irr}}(S^3_{p/q}(3_1\# 3_1))$ as a closed subscheme of $\mathscr{X}_{\text{irr}}(S^3\backslash (3_1\# 3_1))$. First, we have the following lemma.
\begin{lemma}
Let $\varphi: \mathscr{X}_{\text{irr}}(S^3\backslash (3_1\# 3_1))\to \mathscr{X}(S^3\backslash 3_1)\times_{\mathbb{C}}\mathscr{X}(S^3\backslash 3_1)$ denote the map to the fiber product over the meridional trace. Then 
\begin{align*}
\mathscr{X}_{\text{irr}}(S^3_{p/q}(3_1\# 3_1)) = \varphi^{-1}(\varphi( \mathscr{X}_{\text{irr}}(S^3_{p/q}(3_1\# 3_1)))).
\end{align*}
\end{lemma}
\begin{proof}
A character $[\rho]=[(\rho_1,\rho_2)]\in \mathscr{X}_{\text{irr}}(S^3\backslash (3_1\# 3_1))$ is in the character scheme for the $p/q$ surgery if the surgery equation $\rho( m^p\ell^q)=I$ is satisfied. For a composite knot, the longitude $\ell$ is the product of the two longitudes for the constituent knots. Thus, the surgery equation is
\begin{align*}
\rho_1( m)^p\left(\rho_1(\ell_1)\rho_2(\ell_2)\right)^q=I.
\end{align*}
If $[\rho']\in\varphi^{-1}(\varphi([\rho]))$, then it is of the form $[\rho']=[(\rho_1,g^{-1}\rho_2 g)]$ for some $g\in\op{Stab}(\rho(m))$. For an irreducible representation, we cannot have $\rho( m)=\pm I$. Thus, $\op{Stab}(\rho( m))$ is one-dimensional. Furthermore, since $\ell_2$ and $ m$ commute, we must have $\operatorname{Stab}(\rho( m))\subset \op{Stab}(\rho(\ell_2))$. Therefore, $g^{-1}\rho_2(\ell_2) g=\rho_2(\ell_2)$, verifying the surgery equation for $[\rho']$. 
\end{proof}
Thanks to this lemma, it suffices to describe $\varphi(\mathscr{X}_{\text{irr}}(S^3_{p/q}(3_1\# 3_1)))$. We consider each of the three different types of points in $\mathscr{X}_{\text{irr}}(3_1)\times_{\mathbb{C}} \mathscr{X}_{\text{irr}}(3_1)$ separately.
\begin{lemma}
The locus of characters of $\pi_1(S^3_{p/q}(3_1\# 3_1))$ that restrict to an irreducible in $\pi_1(S^3\backslash K_1)$ and an abelian in $\pi_1(S^3\backslash K_2)$ is $\mathscr{X}_{\text{irr}}(S^3_{p/q}(3_1\# 3_1))\cap \varphi^{-1}(\mathscr{X}_{\text{irr}}(M_i))$. This space consists of $\lambda_{\SLC}(S^3_{p/q}(3_1))$ points.
\end{lemma}
\begin{proof}
For $([\rho_1],[\rho_2])\in \mathscr{X}_{\text{irr}}(M_1)\subset \mathscr{X}(3_1)\times_{\mathbb{C}} \mathscr{X}(3_1)$, $\rho_2$ is an abelian representation. Thus, $\rho_2(\ell_2)=I$. The surgery equation then reduces to $\rho( m^p\ell_1^q)=I$, which is just the condition for $p/q$ surgery on the trefoil. So, 
\begin{align*}
|\varphi(\mathscr{X}_{\text{irr}}(S^3_{p/q}(3_1\# 3_1)))\cap \mathscr{X}_{\text{irr}}(M_1)|=\lambda_{\SLC}(S^3_{p/q}(3_1)).
\end{align*}
Since the fibers of $\varphi$ over these types of characters are just points, we obtain the result.
\end{proof}
\begin{lemma}
The set of characters that restrict to an irreducible representation on both factors is given by $\mathscr{X}_{\text{irr}}(S^3_{p/q}(3_1))\cap \varphi^{-1}(\mathscr{X}_{\text{irr}}\times_{\mathbb{C}} \mathscr{X}_{\text{irr}})$, which consists of
\begin{equation*}
\begin{cases}
\lambda_{\SLC}(S^3_{p/2q}(3_1))=\frac{1}{2}|p-12q|-\frac{1}{2} \text{ copies of }\mathbb{C}^* & \text{ if } p \text{ is odd}, \  \\
\frac{1}{2}|p-12q|-1\text{ copies of }\mathbb{C}^* & \text{ if } p \text{ is even}, p\neq 12k, \  \\
\frac{1}{2}|p-12q|-3\text{ copies of }\mathbb{C}^* & \text{ if } p=12k, p/q\neq 12,\  \\
\varphi^{-1}(\mathscr{X}_{\text{irr}}\times_{\mathbb{C}} \mathscr{X}_{\text{irr}})  & \text{ if } p/q=12.
\end{cases}
\end{equation*}
\end{lemma}
\begin{proof}
For irreducible representations of $\pi_1(S^3\backslash 3_1)$, $\rho(\ell)$ is determined by $\rho( m)$. In fact, we have $\rho(\ell)=-\rho( m)^{-6}$. For a point $\varphi([\rho])=([\rho_1],[\rho_2])\in \mathscr{X}_{\text{irr}}\times_{\mathbb{C}} \mathscr{X}_{\text{irr}}$, we have $\rho_1( m_1)=\rho_2( m_2)$, so that $\rho_1(\ell_1)=\rho_2(\ell_2)$. Thus, 
\begin{align*}
\rho( m^p\ell^q)=\rho_1( m^p\ell_1^{2q}).
\end{align*}
For $p$ odd, the equation $\rho_1( m^p\ell_1^{2q})=I$ is just the defining equation for $p/2q$ surgery on the trefoil. Thus, we obtain $\lambda_{\SLC}(S^3_{p/2q}(3_1))$ points. None of these occur at meridional trace $\pm 2$, so that the fiber of $\varphi$ is a copy of $\mathbb{C}^*$ for all of these points.

For $p$ even, $p\neq 12k$, the surgery equation
\begin{align*}
\rho( m)^{p-12q}=I
\end{align*}
has an even exponent. Thus, we obtain 
\begin{align*}
\frac{1}{2}(|12q-p|-2)
\end{align*}
distinct characters, where the $-2$ term serves to discount the roots at $\rho( m)=\pm I$. For $p=12k, p/q\neq 12$, two of the characters in this count occur at meridional trace $\pm\sqrt{3}$, so we subtract $2$ in this case. Again, all of the fibers of $\varphi$ are $\mathbb{C}^*$.

For $p/q=12$, the surgery equation is trivial, so that every representation of this form provides a representation of the surgery.
\end{proof}
\begin{lemma}\label{grannynonab}
The set of irreducible representations formed from a composite of non-abelian reducible representations is 
\begin{equation*}
\mathscr{X}_{\text{irr}}(S^3_{p/q}(3_1\# 3_1))\cap \varphi^{-1}(\mathscr{X}_{\text{nar}}\times_{\mathbb{C}} \mathscr{X}_{\text{nar}}) =
\begin{cases}
2 \text{ copies of }\mathbb{C}^*\backslash\{1\} & \text{ if } p=12k, \  \\
\emptyset & \text{ else. }  \  \\
\end{cases}
\end{equation*}
\end{lemma}
\begin{proof}
For $([\rho_1],[\rho_2])\in \mathscr{X}_{\text{nar}}\times_{\mathbb{C}} \mathscr{X}_{\text{nar}}$, we have $\tr(\rho_i( m))=\pm\sqrt{3}$ and $\rho_i(\ell_i)=I$. Thus, the surgery equation becomes $\rho( m)^p=I$. This holds if and only if $p=12k$.
\end{proof}
For the remaining case of $p/q=12$, we have found that the character scheme of 12 surgery on the granny knot, $\mathscr{X}_{\text{irr}}(S^3_{12}(3_1\# 3_1))$, consists of 2 points coming from the irreducible representation in each of the two copies of $\mathscr{X}_{\text{irr}}(S^3_{12}(3_1))$ and the surface
\begin{align*}
S=\varphi^{-1}(\mathscr{X}_{\text{nar}}\times_{\mathbb{C}}\mathscr{X}_{\text{nar}}\cup\mathscr{X}_{\text{irr}}\times_{\mathbb{C}}\mathscr{X}_{\text{irr}}).
\end{align*}
Putting this and the preceding lemmas together, we obtain Proposition \ref{grannyvar}.
\begin{remark}
12 surgery on the granny knot yields a Seifert fiber space fibered over the orbifold base $S^2(2,2,3,3)$ \cite{kalliongis}. Thus,
\begin{align*}
\pi_1(S_{12}^3(3_1\# 3_1))\cong \langle a,b,c \mid a^3=b^3=c^2=(abc)^{-2}\rangle.
\end{align*}
\end{remark}
\subsection{Character scheme for square knot surgeries}
Let $3_1\# 3_1^*$ denote the square knot, a connected sum of two mirror trefoils, and $S^3_{p/q}(3_1\# 3_1^*)$ the $p/q$ surgery. We have the following description of $\mathscr{X}_{\text{irr}}(S^3_{p/q}(3_1\# 3_1^*))$ 
\begin{proposition}\label{squarevar}
The character scheme $\mathscr{X}_{\text{irr}}(S^3_{p/q}(3_1\# 3_1^*))$ consists of $\lambda_{\SLC}(S^3_{p/q}(3_1))+\lambda_{\SLC}(S^3_{-p/q}(3_1))$ points and
\begin{itemize}
\item $\frac{1}{2}|p|-\frac{1}{2}$ copies of $\mathbb{C}^*$ when $p$ is odd,
\item $\frac{1}{2}|p|-1$ copies of $\mathbb{C}^*$ when $p$ is even, $p\neq 12k$,
\item $\frac{1}{2}|p|-3$ copies of $\mathbb{C}^*$ and 2 copies of $\mathbb{C}^*\backslash\{1\}$ when $p=12k\neq 0$,
\item $S=\varphi^{-1}(\mathscr{X}_{\text{nar}}\times_{\mathbb{C}}\mathscr{X}_{\text{nar}}\cup\mathscr{X}_{\text{irr}}\times_{\mathbb{C}}\mathscr{X}_{\text{irr}})$ when $p=0$.
\end{itemize}
\end{proposition}
\begin{proof}
The proof is analogous to that of Proposition \ref{grannyvar}. The essential difference is that we also need to consider the representations of the left-handed trefoil. Since $S^3_{p/q}(3_1)\cong S^3_{-p/q}(3_1^*)$, we can relate the Casson invariants by
\begin{align*}
\lambda_{\SLC}(S^3_{p/q}(3_1))= \lambda_{\SLC} (S^3_{-p/q}(3_1^*)).
\end{align*}
Thus, the intersection of $\mathscr{X}_{\text{irr}}(S^3_{p/q}(3_1\# 3_1^*))$ with the two copies of $\mathscr{X}_{\text{irr}}(3_1)$ give contributions of $\lambda_{\SLC}(S^3_{p/q}(3_1))$ and $\lambda_{\SLC}(S^3_{-p/q}(3_1))$ points, depending on whether the copy of $\mathscr{X}_{\text{irr}}(3_1)$ corresponds to the right or left-handed trefoil.

For irreducible representations of the right-handed trefoil, we have $\rho_1(\ell_1)=-\rho_1(m)^{-6}$, whereas for the left-handed trefoil we have $\rho_2(\ell_2)=-\rho_2(m)^6$. So, for a representation of the composite that restricts to irreducibles on either factor, we find that $\rho(\ell)=\rho(\ell_1\ell_2)=I$. The equation for $p/q$ surgery reduces to
\begin{align*}
\rho(m)^p=I.
\end{align*}
Throwing away the solutions $\rho(m)=\pm I$ and counting solutions up to conjugacy (i.e. dividing by the equivalence $\rho(m)\sim\rho(m)^{-1}$), we find $\frac{1}{2}|p|-\frac{1}{2}$ solutions for $p$ odd, and $\frac{1}{2}|p|-1$ solutions for $p$ even, $p\neq 12k$. For $p=12k\neq 0$, we omit the two solutions with $\tr(\rho(m))=\pm\sqrt{3}$, as these correspond to non-abelian reducible representations rather than irreducibles. The case of irreducibles formed from the composite of non-abelian reducible representations, which only occurs when $p=12k$, is the same as in Lemma \ref{grannynonab}. When $p=0$, the surgery equation is trivial, and we have the same situation as for $p=12$ for the granny knot.
\end{proof}

\begin{remark}
0 surgery on the square knot yields a Seifert fiber space fibered over the orbifold base $S^2(-2,2,3,3)$ \cite{kalliongis}. Thus,
\begin{align*}
\pi_1(S_{0}^3(3_1\# 3_1^*))\cong \langle a,b,c \mid a^3=b^3=c^2=(abc)^{2}\rangle.
\end{align*}
\end{remark}

\subsection{Smoothness of the Character Schemes}

\begin{proposition}
Let $3_1\# 3_1$ and $3_1\# 3_1^*$ denote the granny and square knots, respectively. The schemes $\mathscr{X}_{\text{irr}}(S^3_{p/q}(3_1\# 3_1))$ and $\mathscr{X}_{\text{irr}}(S^3_{p/q}(3_1\# 3_1^*))$ are smooth schemes for all $p$ and $q$. 
\end{proposition}
\begin{proof}
The sets of complex points of these schemes were computed in the previous section. They consisted of components of dimensions zero, one, and, in the cases of $S^3_{12}(3_1\# 3_1)$ and $S^3_0(3_1\# 3_1^*)$, two. To establish the smoothness of the character scheme near some irreducible representation $\rho$, we must show that the local dimension of the set of complex points at $\rho$ equals the dimension of the tangent space to the scheme at $\rho$. Recall that for an irreducible representation $\rho$ the tangent space is computed by $T_{[\rho]}\mathscr{X}_{\text{irr}}(\Gamma)=H^1(\Gamma;\operatorname{ad}\rho)$. Thus, the proposition follows from the calculation of these $H^1$ groups in Lemma \ref{h1surgery} below, which we prove after two preliminary lemmas.
\end{proof}

\begin{lemma}\label{h1exterior}
Let $\rho$ be an irreducible representation of $\pi_1(S^3\backslash(3_1\#3_1^\circ))$ (where $3_1\#3_1^\circ$ is either the square or granny knot, which have isomorphic fundamental groups). Let $\rho_1$ and $\rho_2$ be the restrictions of $\rho$ to each of the two copies of $\pi_1(S^3\backslash 3_1)$. Then,
\begin{equation*}
\dim H^1(\pi_1(S^3\backslash(3_1\# 3_1^\circ));\operatorname{ad}\rho)=
\begin{cases}
 2 & \text{ if neither of the } \rho_i \text{ are abelian,}  \  \\
 1 & \text{ if either of the } \rho_i \text{ are abelian.}
\end{cases}
\end{equation*}
\end{lemma}
\begin{proof}
We can compute $H^1(\pi_1(S^3\backslash (3_1\#3_1^\circ));\operatorname{ad}\rho)$ (we will suppress the $\pi_1$ from this notation without confusion, as all spaces in consideration are aspherical) from the following portion of the Mayer-Vietoris sequence:
\begin{align}\label{MV1}
\begin{split}
    0&\to H^0(S^3\backslash 3_1;\operatorname{ad}\rho_1)\oplus H^0(S^3\backslash 3_1;\operatorname{ad}\rho_2)\to H^0(S^1\operatorname{ad}\rho)\to H^1(S^3\backslash (3_1\#3_1^\circ)\operatorname{ad}\rho) \to 
\\  & \to H^1(S^3\backslash 3_1;\operatorname{ad}\rho_1)\oplus H^1(S^3\backslash 3_1;\operatorname{ad}\rho_2) \to H^1(S^1;\operatorname{ad}\rho)\to\dots
\end{split}
\end{align}
The $\rho_i$ are the restrictions of $\rho$ to the two copies of $S^3\backslash 3_1$, and the $S^1$ refers to the meridional annulus along which the connected sum operation is performed. Technically, $\rho$ restricts to the complement of the meridional annulus inside of $S^3\backslash 3_1$, but since removing a subset of the boundary of a manifold does not change its homotopy type, this is homotopy equivalent to $S^3\backslash 3_1$ so we ignore the distinction.

Observe that $H^1(S^1;\operatorname{ad}\rho)\cong H^0(S^1;\operatorname{ad}\rho)\cong \mathbb{C}$. The first isomorphism follows from Poincar\' {e} duality. The second follows from the fact that since $\rho$ is an irreducible representation of $\pi_1(S^3\backslash (3_1\#3_1^\circ))$, it restricts to a non-central abelian representation on the meridian and the invariants of such a representation are a one-dimensional subspace of $\operatorname{ad}\rho$.

The last map in \eqref{MV1} is the sum of two maps, each of the form $H^1(S^3\backslash 3_1;\operatorname{ad}\rho_i) \to H^1(S^1;\operatorname{ad}\rho)$. When $\rho_i$ is irreducible, this is the derivative at $[\rho_i]$ of the natural map $\mathscr{X}_{\text{irr}}(S^3\backslash 3_1;\operatorname{ad}\rho_i)\to \mathscr{X}(S^1;\operatorname{ad}\rho)$, where $S^1$ refers to the meridional circle. From our description of $\mathscr{X}_{\text{irr}}(S^3\backslash 3_1)$ as a plane curve, we observe that the meridional trace map is non-singular at all points. Thus, the map on tangent spaces is surjective.

We now consider the case when the $\rho_i$ are both irreducible or both non-abelian reducibles. In this case, $H^0(S^3\backslash 3_1;\operatorname{ad}\rho_i)=0$. When $\rho_i$ is an irreducible representation, we observe that $\dim H^1(S^3\backslash 3_1;\operatorname{ad}\rho_i)=1$ because the character scheme is smooth of dimension 1. When $\rho_i$ is a non-abelian reducible, we can compute $\dim H^1(S^3\backslash 3_1;\operatorname{ad}\rho_i)=1$ directly, as there are only finitely many non-abelian reducible representations up to conjugacy. From this data, \eqref{MV1} yields $\dim H^1(S^3\backslash (3_1\#3_1^\circ);\operatorname{ad}\rho) =2$.

When $\rho_1$ is abelian and $\rho_2$ is irreducible, $H^0(S^3\backslash 3_1;\operatorname{ad}\rho_2)=0$ and the map $H^0(S^3\backslash 3_1;\operatorname{ad}\rho_1)\to H^0(S^1;\operatorname{ad}\rho)$ at the start of \eqref{MV1} is an isomorphism. For an abelian representation, $\dim H^1(S^3\backslash 3_1;\operatorname{ad}\rho_1)=1$. Thus, we compute $\dim H^1(S^3\backslash (3_1\#3_1^\circ);\operatorname{ad}\rho)=1$.
\end{proof}

\begin{lemma}\label{h1surgery}
Let $3_1\# 3_1$ and $3_1\# 3_1^*$ denote the granny and square knots (and let $3_1\#3_1^{\circ}$ denote either). Let $\rho$ be an irreducible representation of $\pi_1(S_{p/q}^3(3_1\#3_1^{\circ}))$. Let $\rho_1$ and $\rho_2$ be the restrictions of $\rho$ to each of the two copies of $\pi_1(S^3\backslash 3_1)$. Then,
\begin{equation*}
\dim H^1(\pi_1(S_{p/q}^3(3_1\#3_1^{\circ}));\operatorname{ad}\rho)=
\begin{cases}
 2 & \text{ if both of the } \rho_i \text{ are non-abelian and } $p/q=12$ \text{ for the granny knot or } \\ & $p/q=0$ \text{ for the square knot,}  \  \\
 1 & \text{ if both of the } \rho_i \text{ are irreducible and we are not in the above case,}  \  \\
 0 & \text{ if either of the } \rho_i \text{ are abelian.}
\end{cases}
\end{equation*}
\end{lemma}
\begin{proof}
We can compute $H^1(S^3_{p/q}(3_1\#3_1^{\circ});\operatorname{ad}\rho)$ from the following Mayer-Vietoris sequence:
\begin{align}\label{MV2}
\dots \overset{0}\rightarrow H^1(S^3_{p/q}(3_1\#3_1^{\circ});\operatorname{ad}\rho)\to H^1(S^3\backslash (3_1\#3_1^{\circ});\operatorname{ad}\rho)\oplus H^1(D^2\times S^1, \operatorname{ad}\rho) \overset{f}\rightarrow H^1(T^2;\operatorname{ad}\rho)\to\dots
\end{align}	
Since $\rho$ must restrict to a non-central abelian representation on the boundary torus, we have $H^2(T^2;\operatorname{ad}\rho)\cong H^0(T^2;\operatorname{ad}\rho)\cong\mathbb{C}$. From the Euler characteristic, we compute $\dim H^1(T^2;\operatorname{ad}\rho)=2$. Similarly, $\rho$ restricts to a non-central abelian representation on the solid torus (if it sent the core of the solid torus to a central element, then in fact $\rho$ would be central on the entire boundary torus, and in particular on the meridian). So, $\dim H^1(D^2\times S^1;\operatorname{ad}\rho)=1$.

We claim that $f$ has rank 1 when $p/q=12$ for the granny knot and $p/q=0$ for the square knot and neither of the $\rho_i$ are abelian representations, and in all other cases, $f$ has rank 2.

Let $s:\mathscr{X}(D^2\times S^1)\to \mathscr{X}(T^2)$ be the restriction map. The map on cohomology groups $H^1(D^2\times S^1; \operatorname{ad}\rho)\to H^1(T^2;\operatorname{ad}\rho)$ can be identified with $ds_{[\rho]}$, the derivative of $s$ at $[\rho]$. Similarly, we can identify the map $H^1(S^3\backslash (3_1\#3_1^{\circ});\operatorname{ad}\rho)\to H^1(T^2;\operatorname{ad}\rho)$ with the derivative at $[\rho]$ of the restriction map $r: \mathscr{X}_{\text{irr}}(S^3\backslash (3_1\#3_1^{\circ}))\to \mathscr{X}(T^2)$. Thus, we can write $f$ as $f=(dr\oplus ds)_{[\rho]}$.

By a standard application of Lefschetz duality and the long exact sequence of the pair $(Y,\partial Y)$, where here $Y=D^2\times S^1$ or $S^3\backslash(3_1\#3_1^{\circ})$, we know that $\operatorname{rank}(dr)=\operatorname{rank}(ds)=1$ \cite{sik}. Thus, the rank of $f$ is $2$ unless the images of $r$ and $s$ have the same tangent spaces at $[\rho]$, in which case the rank of $f$ is $1$. We claim that this equality of tangent spaces occurs only when $p/q=12$ for the granny knot and $p/q=0$ for the square knot and neither of the $\rho_i$ are abelian representations.

Let $t:\mathbb{C}^*\times\mathbb{C}^*\to \mathscr{X}(T^2)$ be the map from the eigenvalue variety to the character variety. With the coordinates $(M,L)$ on $\mathbb{C}^*\times\mathbb{C}^*$ for the meridional and longitudinal eigenvalues, $t(M,L)$ is the class of a representation with $\rho(m)=\operatorname{diag}(M,M^{-1})$ and $\rho(\ell)=\operatorname{diag}(L,L^{-1})$. Away from the central representations, $t$ is a degree two covering map. Thus, we can consider the tangent spaces to $t^{-1}(\operatorname{im}(s))$ and $t^{-1}(\operatorname{im}(r))$ in order to prove the claim.

The curve $t^{-1}(\operatorname{im}(s))$ is the surgery curve $\{M^pL^q=1\}$. The closure of the curve $t^{-1}(\operatorname{im}(r))$ is the vanishing locus of the $A$-polynomial of the knot (ignoring the factor coming from reducibles). Recall our calculation of the $A$-polynomials from Section 6,
\begin{align*}	
A_{3_1\# 3_1}^{\text{irr}}(M,L)= & (L+M^{-6})(L-M^{-12}), \\
A_{3_1\# 3_1^*}^{\text{irr}}(M,L)= & (L+M^{-6})(L+M^{6})(L-1).
\end{align*}

The factor of $L+M^{-6}$ (which is the $A$-polynomial of the right-handed trefoil) comes from representations that are irreducible on a $3_1$ summand and abelian on the other summand. Similarly, $L+M^{6}$ is the A-polynomial of the left-handed trefoil. The last factors come from the composites of two non-abelian representations. For such representations of the the granny knot, we have $L_1=L_2=-M^{-6}$ and $L=L_1L_2$, so that $L=M^{-12}$. For the square knot, $L_1=L_2^{-1}$, so that this component is mapped to the line $L=1$.

Now we see that the only situations in which the tangent space to the vanishing locus of the $A$-polynomial coincides with the tangent space to the surgery curve are when $p=12,q=1$ for the granny knot or $p=0,q=1$ for the square knot and $\rho$ is a composite of two non-abelian representations $\rho_i$. This proves the claim.

From \eqref{MV2}, we see that
\begin{align*}
& \dim H^1(S^3_{p/q}(3_1\#3_1^{\circ});\operatorname{ad}\rho) = \dim H^1(S^3\backslash (3_1\#3_1^{\circ});\operatorname{ad}\rho)+1-\operatorname{rank}(f).
\end{align*}
The result follows from combining the above formula, our computations of the rank of $f$, and Lemma \ref{h1exterior}.
\end{proof}
Theorems \ref{HPgranny} and \ref{HPsquare} now follow from applying Corollary \ref{modtwo} to the calculation of the respective character varieties in Propositions \ref{grannyvar} and \ref{squarevar} and the determination of the singular cohomology of these character schemes from Proposition \ref{cohomology}.
\begin{remark}
We use $\HP$ with $\mathbb{Z}/2\mathbb{Z}$ coefficients in Theorems \ref{HPgranny} and \ref{HPsquare} only to avoid determining the relevant local system. Indeed, the character schemes of surgeries on $3_1\# 3_1$ include some components isomorphic to $\mathbb{C}^*$ and $\mathbb{C}^*\backslash\{1\}$, while the other topological types of components that appear are simply connected. We conjecture that the local systems are in fact trivial on all of the components and that Theorems \ref{HPgranny} and \ref{HPsquare} hold over $\mathbb{Z}$.
\end{remark}

\section{Further Discussion}

\subsection{Exact triangles}
In analogy with other Floer theories \cite{OS}\cite{scaduto}\cite{floer}, one may conjecture the existence of a surgery exact triangle for $\HP_{\#}$. That is, one may hope that there exists a long exact sequence
\begin{align*}
\HP_{\#}(S^3)[1]\to \HP_{\#}(S^3_{p+1}(K)) \to \HP_{\#}(S^3_{p}(K))\to\HP_{\#}(S^3).
\end{align*}
However, since $\HP_{\#}(S^3)$ is supported in degree zero, such a long exact sequence would imply that $\HP_{\#}(S^3_{p}(K))$ and $\HP_{\#}(S^3_{p+1}(K))$ are isomorphic except possibly in degree zero. Yet the data from Proposition \ref{HPframed2bridge} shows that this is not the case for surgeries on two-bridge knots. For example, if $p=2k$ and $p+1=2k+1$ both satisfy the hypotheses of Proposition \ref{HPframed2bridge}, then $\HP_{\#}(S^3_{p}(K))$ has rank $k-1$ in degree $-2$, whereas $\HP_{\#}(S^3_{p+1}(K))$ has rank $k$ in degree $-2$.

One can also ask whether a surgery exact triangle exists for $\HP$. The data in Proposition \ref{HP2bridge} can be used to show that such a triangle cannot exist for two-bridge knots. However, one would not even expect such a surgery exact triangle for $\HP$ since exact triangles in Floer theories are not usually formulated for the versions that exclude reducibles. For example, there is no surgery exact triangle for $\mathit{HF}_{\operatorname{red}}^{\circ}$ in Heegaard Floer homology.

\subsection{A conjecture}
In \cite{BC}, the authors define an $\SLC$ Casson knot invariant by
\begin{align*}
\lambda_{\SLC}'(K)=\lim\limits_{q\to\infty} \frac{1}{q}\lambda_{\SLC}(S^3_{p/q}(K)),
\end{align*}
where $p$ is fixed and the limit is taken over all $q$ relatively prime to $p$. In particular, this quantity is independent of $p$. We can make the analogous conjecture for $\HP$ and $\HP_{\#}$.
\begin{conjecture}
Let $K\subset S^3$ be a knot and $S^3_{p/q}(K)$ it $p/q$ surgery. Then the quantities 
\begin{align*}
\lim\limits_{q\to\infty} \frac{1}{q}\operatorname{rk}(\HP^n(S^3_{p/q}(K)))
\end{align*}
and
\begin{align*}
\lim\limits_{q\to\infty} \frac{1}{q}\operatorname{rk}(\HP_{\#}^n(S^3_{p/q}(K)))
\end{align*}
are well-defined invariants of the knot $K$. 
\end{conjecture}
For example, by Theorems \ref{HPgranny} and \ref{HPsquare} we can verify this conjecture for $\HP$ of surgeries on the granny and square knots. We obtain the numerical data
\begin{align*}
\lim\limits_{q\to\infty} \frac{1}{q}\operatorname{rk}(\HP^0(S^3_{p/q}(3_1\# 3_1)))=12,\\
\lim\limits_{q\to\infty} \frac{1}{q}\operatorname{rk}(\HP^{-1}(S^3_{p/q}(3_1\# 3_1)))=6,
\end{align*}
and
\begin{align*}
\lim\limits_{q\to\infty} \frac{1}{q}\operatorname{rk}(\HP^0(S^3_{p/q}(3_1\# 3_1^*)))=6,\\
\lim\limits_{q\to\infty} \frac{1}{q}\operatorname{rk}(\HP^{-1}(S^3_{p/q}(3_1\# 3_1^*)))=0.
\end{align*}

\end{document}